\documentclass[10pt]{amsart}
\usepackage[english]{babel}
\usepackage{amsmath,amsthm}
\usepackage{amsfonts}
\usepackage{enumerate}
\usepackage{graphicx}
\usepackage{color}
\usepackage[all]{xy}

\usepackage[top=1.5in, bottom=1.5in, left=1.5in, right=1.5in]{geometry}

\newtheorem{thm}{Theorem}[section]
\newtheorem{cor}[thm]{Corollary}
\newtheorem{lem}[thm]{Lemma}
\newtheorem{prop}[thm]{Proposition}
\newtheorem{conj}[thm]{Conjecture}
\newtheorem*{claim}{Claim}

\newtheorem{innercustomthm}{}
\newenvironment{customthm}[1]
  {\renewcommand\theinnercustomthm{#1}\innercustomthm}
  {\endinnercustomthm}

\theoremstyle{definition}

\theoremstyle{remark}
\newtheorem{rem}[thm]{Remark}

\numberwithin{equation}{section}

\newcommand{\spin}{\ifmmode{\rm Spin}\else{${\rm spin}$\ }\fi}
\newcommand{\spinc}{\ifmmode{{\rm Spin}^c}\else{${\rm spin}^c$}\fi}

\newcommand{\spincs}{\mathfrak s}

\newcommand{\hfhat}{\widehat{HF}}
\newcommand{\hf}{HF}
\newcommand{\hfp}{HF^+}
\newcommand{\hfred}{HF_{\rm red}}

\newcommand{\hfkhat}{\widehat{HFK}}
\newcommand{\ared}[1]{\mathbf{A}_{{\rm red},#1}}

\newcommand{\tp}{\mathcal{T}^+}

\newcommand{\Xipq}{\mathbb{X}^+_{i,p/q}}

\newcommand{\Abold}{\mathbf{A}^+}
\newcommand{\Bbold}{\mathbf{B}^+}
\newcommand{\AT}{\mathbf{A}^T}
\newcommand{\vbold}{\mathbf{v}}
\newcommand{\hbold}{\mathbf{h}}
\newcommand{\Dbold}{\mathbf{D}_{i,p/q}^+}

\newcommand{\Ap}{\mathbb{A}^+}
\newcommand{\Bp}{\mathbb{B}^+}

\newcommand{\Z}{\mathbb{Z}}
\newcommand{\F}{\mathbb{F}}
\newcommand{\Q}{\mathbb{Q}}

\newcommand{\rk}{{\rm rk}}

\newcommand{\Zp}{\mathbb{Z}/p\mathbb{Z}}

\begin{document}

\title{Non-integer characterizing slopes for torus knots}%
\author{Duncan McCoy}%
\address {Department of Mathematics, University of Texas at Austin, Austin, TX 78712}
\email{d.mccoy@math.utexas.edu}
\date{}%

\begin{abstract}
A slope $p/q$ is a characterizing slope for a knot $K$ in $S^3$ if the oriented homeomorphism type of $p/q$-surgery on $K$ determines $K$ uniquely. We show that for each torus knot its set of characterizing slopes contains all but finitely many non-integer slopes. This generalizes work of Ni and Zhang who established such a result for $T_{5,2}$. Along the way we show that if two knots $K$ and $K'$ in $S^3$ have homeomorphic $p/q$-surgeries, then for $q\geq 3$ and $p$ sufficiently large we can conclude that $K$ and $K'$ have the same genera and Alexander polynomials. This is achieved by consideration of the absolute grading on Heegaard Floer homology.
\end{abstract}
\maketitle

\section{Introduction}
Given a knot $K\subseteq S^3$, we say that $p/q\in \Q$ is a {\em characterizing slope} for $K$, if the oriented homeomorphism type of the manifold obtained by $p/q$-surgery on $K$ determines $K$ uniquely.\footnote{Throughout the paper, we use $Y'\cong Y$ to denote the existence of a orientation-preserving homeomorphism between $Y$ and $Y'$.} In general determining the set of characterizing slopes for a given knot is challenging. It was a long-standing conjecture of Gordon, eventually proven by Kronheimer, Mrowka, Ozsv{\'a}th and Szab{\'o}, that every slope is a characterizing slope for the unknot \cite{kronheimer07lensspacsurgeries}. Ozsv{\'a}th and Szab{\'o} have also shown that every slope is a characterizing slope for the trefoil and the figure-eight knot \cite{Ozsvath2006trefoilsurgery}. More recently, Ni and Zhang have studied characterizing slopes for torus knots, showing that $T_{5,2}$ has only finitely many non-characterizing slopes which are not negative integers \cite{Ni14characterizing}. The aim of this paper is to establish a similar result for arbitrary torus knots. We will be primarily interested in non-integer surgery slopes. For each torus knot, the main result of this paper is to classify the non-integer non-characterizing slopes outside of a finite set of slopes.

\begin{thm}\label{thm:toruscharslopes}
For $s>r>1$ and $q\geq 2$ let $K$ be a knot such that $S_{p/q}^3(K)\cong S_{p/q}^3(T_{r,s})$. If $p$ and $q$ satisfy at least one of the following:
\begin{enumerate}[(i)]
\item $p \leq \min \{-\frac{43}{4}(rs-r-s),-32q\}$,
\item $p\geq \max\{\frac{43}{4}(rs-r-s), 32q+ 2q(r-1)(s-1)\}$, or
\item $q\geq 9$,
\end{enumerate}
then we have either $(a)$ $K=T_{r,s}$, or $(b)$ $K$ is a cable of a torus knot, in which case $q=\lfloor s/r\rfloor$, $p=\frac{r^2 q^4 -1}{q^2-1}$, $s=\frac{rq^3 \pm 1}{q^2-1}$ and $r>q$.
\end{thm}
When combined with previously known results about integer characterizing slopes for torus knots, this yields the following corollary \cite{mccoy2014sharp,Ni14characterizing}.
\begin{cor}\label{cor:intcase}
The knot $T_{r,s}$ with $r,s>1$ has only finitely many non-characterizing slopes which are not negative integers.
\end{cor}

It is well-known that the manifolds obtained by non-integer surgery on torus knots are Seifert fibred spaces \cite{Moser71elementary}. Conjecturally, the only knots in $S^3$ with non-integer Seifert fibred surgeries are torus knots and cables of torus knots.
\begin{conj}\label{conj:SFsurgeries}
If $S_{p/q}^3(K)$ is a Seifert fibred space and $q\geq 2$, then $K$ is either a torus knot or a cable of a torus knot.
\end{conj}
We can use this to obtain a precise conjecture on which non-integer slopes are characterizing slopes for torus knots. In particular, it turns out that conjecturally each torus knot has at most one non-integer non-characterizing slope, which is precisely the non-characterizing slope occurring in the conclusion of Theorem~\ref{thm:toruscharslopes}.

\begin{conj}\label{conj:listofslopes}
For the torus knot $T_{r,s}$ with $s>r>1$, every non-integer slope is characterizing with the possible exception of $p/q$ for $p= \frac{r^2 q^4 -1}{q^2-1}$ and $ q=\lfloor s/r \rfloor\geq 2$, which is non-characterizing only if $r>q$ and $s=\frac{rq^3 \pm 1}{q^2-1}$. Moreover, for this slope there is a unique knot $K \ne T_{r,s}$ with $S_{p/q}^3(K)\cong S_{p/q}^3(T_{r,s})$ and $K$ is a cable of a torus knot.
\end{conj}
One deduces Conjecture~\ref{conj:listofslopes} from Conjecture~\ref{conj:SFsurgeries} by checking when a torus knot can share a surgery with a torus knot or a cable of a torus knot. The classification of when a cable of a torus knot and a cable of a torus knot have a common non-integer surgery is recorded in the following proposition. Note that it shows the converse to Conjecture~\ref{conj:listofslopes} is true: for each $T_{r,s}$ satisfying the necessary conditions, there is a cable of a torus knot exhibiting that the required slope is non-characterizing.
\begin{prop}\label{prop:cablenoncharslopes}
For $s>r>1$ and $q\geq 2$, there exists a non-trivial cable of a torus knot $K$ such that $S_{p/q}^3(K)\cong S_{p/q}^3(T_{r,s})$ if and only if
\[s=\frac{rq^3 \pm 1}{q^2-1},\quad p= \frac{r^2 q^4 -1}{q^2-1}, \quad q=\lfloor s/r\rfloor \quad\text{and}\quad r>q,\]
in which case $K$ is the $(q,\frac{q^2r^2-1}{q^2-1})$-cable of $T_{r,\frac{rq\pm 1}{q^2-1}}$.
\end{prop}

\begin{rem}
If one also allows orientation reversing homeomorphisms in the definition of a characterizing slope, then the list given by Conjecture~\ref{conj:listofslopes} would be incomplete. For example, we have
\[S_{\frac{29}{2}}^3(T_{2,7})\cong -S_{\frac{29}{2}}^3(T_{3,5}).\]
\end{rem}
The proof of Theorem~\ref{thm:toruscharslopes} follows a similar outline to Ni and Zhang's work. Given a knot $K\subseteq S^3$ such that $S_{p/q}^3(K)\cong S_{p/q}^3(T_{r,s})$, we consider the possibilities that $K$ is a hyperbolic knot, a satellite knot or a torus knot in turn. By applying results from hyperbolic geometry and Heegaard Floer homology, we will show that for the slopes in Theorem~\ref{thm:toruscharslopes} the only possibilities are that $K$ is a cable of a torus knot or $K=T_{r,s}$. The bound $q\geq 9$ arises as a result of Lackenby and Meyerhoff's bound on the distance between exceptional surgery slopes \cite{lackenby13exceptional}. The other bounds are a consequence of combining restrictions on exceptional surgeries coming from the 6-theorem of Agol \cite{Agol00BoundsI} and Lackenby \cite{Lackenby03Exceptional} with genus bounds on $K$ coming from Heegaard Floer homology. These genus bounds are the key technical results developed in this paper. In general, we show that if $S_{p/q}^3(K)\cong S_{p/q}^3(K')$, then under certain circumstances $K$ and $K'$ must have the same genera and Alexander polynomials. For arbitrary knots in $S^3$, we have the following result.
\begin{thm}\label{thm:HFKrecovery}
Let $K,K'\subseteq S^3$ be knots such that $S_{p/q}^3(K)\cong S_{p/q}^3(K')$. If
\[|p| \geq 12+4q^2-2q +4qg(K)
\quad \text{and} \quad q\geq 3,\]
then $\Delta_K(t)=\Delta_{K'}(t)$, $g(K)=g(K')$ and $K$ is fibred if and only if $K'$ is fibred.
\end{thm}
Here $\Delta_K(t)$ denotes the Alexander polynomial of $K$. We obtain stronger results for $L$-space knots.\footnote{An $L$-space knot is one with {\em positive} $L$-space surgeries.}
\begin{thm}\label{thm:technical2}
Suppose that $K$ is an $L$-space knot. If $S_{p/q}^3(K)\cong S_{p/q}^3(K')$ for some $K'\subseteq S^3$ and either
\begin{enumerate}[(i)]
\item $p\geq 12+4q^2 -2q +4qg(K)$ or
\item $p\leq \min\{2q-12-4q^2, -2qg(K)\}$ and $q\geq 2$
\end{enumerate}
holds, then $\Delta_K(t)=\Delta_{K'}(t)$, $g(K)=g(K')$ and $K'$ is fibred.
\end{thm}
Both Theorem~\ref{thm:HFKrecovery} and Theorem~\ref{thm:technical2} are proven by making use of the absolute grading in Heegaard Floer homology.
\begin{rem}
Baker and Motegi have recently constructed infinite families of knots $\{K_n\}_{n\in \Z}$ in $S^3$ such that
\[S_n^3(K_0)\cong S_n^3(K_n),\]
for all $n\in \Z$ and $\deg \Delta_{K_n}(t) \rightarrow \infty$ as $|n|\rightarrow \infty$ \cite[Section~3]{baker16characterizing}. This shows that Theorem~\ref{thm:HFKrecovery} cannot be extended unconditionally to integer surgeries.
\end{rem}

\subsection*{Acknowledgements}
The author would like to thank his supervisor, Brendan Owens, for his helpful guidance. He also wishes to acknowledge the influential role of the work of Yi Ni and Xingru Zhang which provided inspiration for both the overall strategy and several technical steps in the proof of Theorem~\ref{thm:toruscharslopes} \cite{Ni14characterizing}. 

\section{Heegaard Floer homology}
Heegaard Floer homology is a package of 3-manifold invariants introduced by Ozsv{\'a}th and Szab{\'o} \cite{Ozsvath04threemanifoldinvariants}. To each closed oriented 3-manifold $Y$ equipped with a \spinc-structure $\spincs$ it associates a family of groups denoted by $\hfhat(Y,\spincs)$, $\hf^\pm(Y,\spincs)$ and $\hf^\infty(Y,\spincs)$.  Throughout this paper all Heegaard Floer groups are taken with $\F=\Z/2\Z$ coefficients.

We will be primarily concerned with $\hfp(Y,\spincs)$, where $Y$ is a rational homology sphere. In this case, the group $\hfp(Y,\spincs)$ possesses an absolute $\Q$-grading. There is also a $U$-action on $\hfp(Y,\spincs)$, which gives $\hfp(Y,\spincs)$ the structure of an $\F[U]$-module. Multiplication by $U$ interacts with the $\Q$-grading by decreasing it by 2 \cite{Ozsvath03Absolutely}.

In any \spinc-structure $\hfp(Y,\spincs)$ can be decomposed as a direct sum:
\[\hfp(Y,\spincs) \cong \tp \oplus \hfred(Y,\spincs),\]
where $\tp=\F[U,U^{-1}]/U\F[U]$ and $U^N\hfred(Y,\spincs)=0$ for $N$ sufficiently large. The $\tp$ summand is sometimes referred to as the {\em tower}. The minimal $\Q$-grading over all elements of the tower is an invariant of $(Y,\spincs)$ called the {\em $d$-invariant} and is denoted $d(Y,\spincs)$. We say that $Y$ is an {\em $L$-space} if
$\hfred(Y,\spincs)=0$ for all $\spincs \in \spinc(Y)$.

Heegaard Floer homology is invariant under conjugation of \spinc-structures, in the sense that $\hfp(Y,\spincs)$ and $\hfp(Y, \overline\spincs)$ are isomorphic as $\F[U]$-modules and as $\Q$-graded groups. In particular, the $d$-invariants satisfy $d(Y,\spincs)=d(Y,\overline\spincs)$.

\subsection{Knot Floer homology}
Knot Floer homology was defined independently by Ozsv{\'a}th and Szab{\'o} \cite{Ozsvath04knotinvariants} and Rasmussen \cite{Rasmussen03Thesis}. Given a knot in $K\subseteq S^3$, it takes the form of a finitely-generated group
\[\hfkhat(K)=\bigoplus_{s\in\Z} \hfkhat(K,s),\]
where $s$ is known as the {\em Alexander grading}. The knot Floer homology also possesses a second grading, known as the {\em Maslov grading} such that
\[\hfkhat(K,s)=\bigoplus_{d\in\Z} \hfkhat_d(K,s).\]
If $K$ has Alexander polynomial
\[\Delta_K(t)=\sum_{s\in \Z} a_s t^s,\]
normalized so that $a_s=a_{-s}$ and $\Delta_K(1)=1$, then $\Delta_K(t)$ can be recovered by taking the Euler characteristic in each Alexander grading:
\[a_s = \chi(\hfkhat(K,s))= \sum_{d\in \Z} (-1)^d \rk\hfkhat_d(K,s). \]
With this normalization in place, we will take $t_k(K)$ to denote the torsion coefficient
\[t_k(K)=\sum_{i\geq 0}ia_{k+i}.\]
\begin{rem}\label{rem:Alexpolydetermined}
The coefficients of $\Delta_K(t)$ satisfy
\[a_k=t_{k+1}(K)-2t_k(K)+t_{k-1}(K)\]
for all $k$. Since the Alexander polynomial is normalized so that $\Delta_K(1)=1$, this means the Alexander polynomial can be computed from the $t_k(K)$ for $k\geq 0$.
\end{rem}

One key geometric property detected by knot Floer homology is the genus \cite{Ozsvath04genusbounds}:
\[g(K)=\max \{s \,|\, \hfkhat(K,s)\ne 0\}.\]
The other geometric property of knot Floer homology that we will use is its ability to detect whether a knot is fibred \cite{Ghiggini08fibredness,Ni07fibredness}.
\begin{thm}[Ni]\label{thm:fiberedness}
A knot $K$ of genus $g$ is fibred if and only if $\rk \hfkhat(K,g)=1$.
\end{thm}

\subsection{The knot Floer chain complex}
The knot Floer homology group $\hfkhat(K)$ can be generalized to the knot Floer chain complex $CFK^\infty(K)$, which takes the form of a bifiltered chain complex
\[CFK^\infty(K)=\bigoplus_{i,j\in \Z}C\{(i,j)\},\]
where $H_*(C\{(i,j)\})\cong \hfkhat_{*-2i}(K,j-i)$.

There is also a natural chain complex isomorphism
\[U: CFK^\infty(K)\longrightarrow CFK^\infty(K),\]
which maps $C\{(i,j)\}$ isomorphically to $C\{(i-1,j-1)\}$ and lowers the Maslov grading by 2. This gives $CFK^\infty(K)$ the structure of a finitely-generated $\F[U,U^{-1}]$-module.

The chain homotopy type of $CFK^\infty(K)$ as a bifiltered complex is an invariant of $K$. In fact, after a suitable chain homotopy, one can assume that
\[C\{(i,j)\}\cong \hfkhat_{*-2i}(K,j-i).\]

The knot Floer complex has several important quotient complexes: for each $k\in \Z$, the ``hook'' complexes
\[A_k^+=C\{i\geq 0 \text{ or } j\geq k\},\]
and the complex
\[B^+=C\{i\geq 0\}.\]
These complexes admit chain maps
\[v_k,h_k \colon A_k^+ \longrightarrow B^+,\]
where $v_k$ is the obvious vertical projection, and $h_k$ consists of the composition of a horizontal projection onto $C\{j\geq k\}$, multiplication by $U^k$ and a chain homotopy equivalence. We will use $\Abold_k=H_*(A_k^+)$ and $\Bbold=H_*(B^+)$ to denote the homology groups and $\mathbf{v}_k$ and $\mathbf{h}_k$ to denote the maps induced on homology by $v_k$ and $h_k$ respectively. As we are working with a knot in $S^3$, we have $\Bbold\cong \hfp(S^3)\cong \tp$. The group $\Abold_k$ stabilizes under multiplication by large powers of $U$, allowing us to define $\AT_k$ as $\AT_k=U^N \Abold_k$ for $N$ sufficiently large. This always satisfies $\AT_k\cong \tp$. We also define $\ared{k}$ to be the quotient  $\ared{k}=\Abold_k/\AT_k$. When restricted to $\AT_k$ the map $\vbold_k$ is modeled on multiplication by $U^{V_k}$ for some non-negative integer $V_k$ \cite{ni2010cosmetic}. Similarly, $\hbold_k$ is modeled on multiplication by $U^{H_k}$ for some non-negative integer $H_k$ when restricted to $\AT_k$. These integers $V_k$ and $H_k$ are known to satisfy
\[V_k=H_{-k} \quad \text{and} \quad V_k-1 \leq V_{k+1} \leq V_{k},\]
for all $k$.

For any $n\geq 0$, we will use $\mathcal{T}(n)$ to denote the $\F[U]$-submodule of $\tp$ generated by $U^{1-n}$. For $n=0$, we take $\mathcal{T}(0)=0$.

The following proposition shows how the Alexander polynomial, genus and fiberedness of $K$ are encoded in the $V_k$ and $\ared{k}$.
\begin{prop}[Cf. Lemma~3.3 of \cite{Ni14characterizing}]\label{prop:aredproperties}
For $K\subseteq S^3$ the following hold:
\begin{enumerate}[(i)]
\item $t_k(K)= V_k + \chi(\ared{k})$ for all $k$;
\item $g(K)=1+\max\{k\,|\, V_k + \rk \ared{k}>0\}$; and
\item $K$ is fibred if and only if $V_{g-1} + \rk \ared{g-1}=1$.
\end{enumerate}
\end{prop}
\begin{proof}
It follows from the definition of $v_k$ and $\ared{k}$ that the kernel of $\vbold_k$ admits a splitting as
\[\ker \vbold_k \cong \mathcal{T}(V_k) \oplus \ared{k}.\]
This shows that
\[\chi(\ker \vbold_k)=V_k + \chi(\ared{k}).\]
On the other hand we have the long exact sequence of chain complexes
\[0\longrightarrow C\{i<0,j\geq k\} \longrightarrow A^+_k \overset{v_k}\longrightarrow B^+ \longrightarrow 0.
\]
As $\vbold_k$ is surjective on homology, the exact triangle induced by this sequence shows that
\[\ker \vbold_k \cong H_*(C\{i<0,j\geq k\}).\]
Taking the Euler characteristic this shows that
\[\chi(\ker \vbold_k)=\sum_{\substack{ i\leq -1,\\ j\geq k}} \chi(\hfkhat(K,j-i))= \sum_{i\geq 1} i \chi(\hfkhat(K,k+i))=t_k(K).\]
This proves $(i)$.

Since $C\{i<0,j\geq g(K)\}=0$, we have $\ker \vbold_{k}=0$ for $k\geq g$. Furthermore, as
\[C\{i<0,j= g(K)\}=C\{-1,g-1\}\cong \hfkhat(K,g),\]
 we have $\ker \vbold_{g-1}\cong \hfkhat(K,g)$. This shows
\[g(K)-1=\max\{k\,|\, V_k + \rk \ared{k}>0\},\]
proving $(ii)$. As Theorem~\ref{thm:fiberedness} shows that $K$ is fibred if and only if $\rk \hfkhat(K,g)=1$, this also proves $(iii)$.
\end{proof}

Let $\nu^+(K)$ to be the quantity
\[\nu^+(K)=\min \{k \,|\, V_k=0\}.\]
It follows from Proposition~\ref{prop:aredproperties} that $\nu^+(K)$ exists and is at most $g(K)$. In fact, it can be shown that $\nu^+(K)\leq g_4(K)$, where $g_4(K)$ is the smooth slice genus of $K$ \cite[Theorem~2.3]{Rasmussen04Goda}. Recall that $K$ is said to be an {\em $L$-space knot} if $S_{p/q}^3(K)$ is an $L$-space for some $p/q>0$. Equivalently, $K$ is an $L$-space knot if and only if $\ared{k}=0$ for all $k$. The following proposition summarizes the properties of $L$-space knots that we require.
\begin{prop}\label{prop:Lspaceknot}
If $K$ is an $L$-space knot, then
\begin{enumerate}[(i)]
\item $\nu^+(K) = g(K)$,
\item $K$ is fibred,
\item $\nu^+(\overline{K})=0$ and
\item $\ared{k}(\overline{K})\cong \mathcal{T}(V_{|k|}(K))$ for all $k$.
\end{enumerate}
\end{prop}
\begin{proof}
As an $L$-space knot satisfies $\ared{k}(K)=0$ for all $k$, it follows from Proposition~\ref{prop:aredproperties}, that $V_{g(K)-1}>0$ and $V_{g(K)}=0$. This shows that $\nu^+(K) = g(K)$. Since $V_{g(K)-1}\leq V_{g(K)}+1$, it follows that $V_{g(K)-1}=1$ and that $K$ is fibred by Proposition~\ref{prop:aredproperties}. The facts about $\overline{K}$ follow from Lemma~16 and Proposition~17 in \cite{gainullin2014mapping}.
\end{proof}

\subsection{The mapping cone formula}
Given a knot in $K \subseteq S^3$, one can determine the Heegaard Floer homology of all manifolds obtained by surgery on it in terms of the knot Floer homology of $K$ via the homology of a mapping cone \cite{ozsvath2011rationalsurgery}. In this section we summarize the results arising from the mapping cone formula that we will need. More detailed accounts of the mapping cone formula and its consequences can be found in \cite{ni2010cosmetic} or \cite{gainullin2014mapping}, for example.

In order to describe the Heegaard Floer homology of $S_{p/q}^3(K)$, we need a way to label its \spinc-structures. This labeling takes the form of an affine bijection defined in terms of relative \spinc-structures on $S^3\setminus \nu K$, \cite{ozsvath2011rationalsurgery}:
\begin{equation}\label{eq:spinccorrespondence}
\phi_{K,p/q} \colon \Zp \longrightarrow \spinc(S_{p/q}^3(K)).
\end{equation}
The exact details of this map are not important, however we note that for any knot $K$, conjugation of \spinc-structures is given by \cite[Lemma~2.2]{ni13halfinteger}:
\begin{equation}\label{eq:conjugation}
\phi_{K,p/q}(q-1-i \bmod{p})=\overline{\phi_{K,p/q}(i)}\in \spinc(S_{p/q}^3(K)).
\end{equation}
If $S_{p/q}^3(K)\cong S_{p/q}^3(K')\cong Y$, then $\phi_{K,p/q}$ and $\phi_{K',p/q}$ will, in general, give rise to different labelings on $\spinc(Y)$. However, as long it will not cause confusion, we will suppress the map $\phi_{K,p/q}$ from the notation.

\paragraph{} When $K$ is the unknot this gives a labeling on the \spinc-structures of a lens space. We will use $d(p,q,i)$ to denote the $d$-invariant
\[d(p,q,i)=d(S_{p/q}^3(U),i) \quad \text{for $i\in \Zp$.}\]

\paragraph{} Now we describe how $CFK^\infty(K)$ determines $\hfp(S_{p/q}^3(K))$. Consider the groups
\[\Ap_i=\bigoplus_{s\in \Z}(s,\Abold_{\lfloor\frac{ps+i}{q}\rfloor}) \quad \text{and} \quad \Bp_i=\bigoplus_{s\in \Z}(s,\Bbold),\]
and the maps
\[\vbold_{\lfloor\frac{ps+i}{q}\rfloor}\colon (s,\Abold_{\lfloor\frac{ps+i}{q}\rfloor}) \rightarrow (s,\Bbold) \quad \text{and} \quad
\hbold_{\lfloor\frac{ps+i}{q}\rfloor}\colon (s,\Abold_{\lfloor\frac{ps+i}{q}\rfloor}) \rightarrow (s+1,\Bbold),\]
where $\vbold_k$ and $\hbold_k$ are the maps on homology induced by $v_k$ and $h_k$ as in the previous section. These maps can be added together to obtain a chain map
\[\Dbold\colon \Ap_i \rightarrow \Bp_i,\]
where
\[
\Dbold(s,x)= (s,\vbold_{\lfloor\frac{ps+i}{q}\rfloor}(x))+ (s+1,\hbold_{\lfloor\frac{ps+i}{q}\rfloor}(x)).
\]
The group $\hfp(S_{p/q}^3(K),i)$ is computed in terms of the mapping cone on $\Dbold$.
\begin{thm}[Ozsv{\'a}th-Szab{\'o}, \cite{ozsvath2011rationalsurgery}]\label{thm:mappingcone}
For any knot $K$ in $S^3$. Let $\Xipq$ be the mapping cone of $\Dbold$, then there is a graded isomorphism of groups
\[H_*(\Xipq) \cong \hfp(S_{p/q}^3(K),i).\]
\end{thm}
\begin{rem}
The statement of Theorem~\ref{thm:mappingcone} given here is not quite the one given in \cite{ozsvath2011rationalsurgery}. Ozsv{\'a}th-Szab{\'o} establish an isomorphism between Heegaard Floer homology and the mapping cone of a map $D_{i,p/q}^+$, whose induced map on homology is $\Dbold$. For surgeries on $S^3$, both mapping cones compute the same Heegaard Floer homology groups.
\end{rem}
\begin{rem}
The isomorphism in Theorem~\ref{thm:mappingcone} is $U$-equivariant, so it also provides an isomorphism of $\F[U]$-modules.
\end{rem}

When $p/q>0$, the map $\Dbold$ is surjective, so Theorem~\ref{thm:mappingcone} gives a graded isomorphism  $\hfp(S_{p/q}^3(K),i)\cong \ker \Dbold$. The grading on $\ker \Dbold$ is determined by putting a $\Q$-grading on $\Xipq$ in such a way that $\Dbold$ decreases the grading by one and the grading on $\Bp_i$, which is independent of $K$, is fixed to give the correct $d$-invariants for surgery on the unknot (cf. \cite[Section~7.2]{ozsvath2011rationalsurgery}).
In practice, this means that for $p/q>0$ and $0\leq i \leq p-1$, the grading on $\Bp_i$ satisfies \cite{ni2010cosmetic}

\begin{equation}\label{eq:basegrading}
gr(0,1)=d(p,q,i)-1
\end{equation}
and, as
\begin{equation*}
H_{-k}(U)=V_k(U)=
\begin{cases}
0 &\text{if $k\geq 0$}\\
|k| &\text{if $k\leq 0$},
\end{cases}
\end{equation*}
the gradings of $(s,1)$ and $(s+1,1)$ in $\Bp_i$ are related by \cite[Section~3.3]{Ni14characterizing}
\begin{equation}\label{eq:towergradshift}
gr(s+1, 1)=gr(s,1)+2\lfloor\frac{i+ps}{q}\rfloor, \quad \text{for any $s\in\Z$.}
\end{equation}

With these gradings one finds that for any $p/q>0$ and any $0\leq i \leq p-1$, the $d$-invariants $S_{p/q}^3(K)$ can be calculated by \cite[Proposition~1.6]{ni2010cosmetic}
\begin{equation}\label{eq:NiWuformula}
d(S_{p/q}^3(K),i)=d(p,q,i)-2\max\{V_{\lfloor\frac{i}{q}\rfloor},V_{\lceil\frac{p-i}{q}\rceil}\}.
\end{equation}
One can also compute the reduced Heegaard Floer homology groups. We require only the special case when $p/q\geq 2\nu^+(K)-1$. The following proposition can easily be derived from \cite[Corollary~12]{gainullin2014mapping} or \cite[Proposition~3.6]{Ni14characterizing}.
\begin{prop}\label{prop:redHFp}
If $p/q\geq 2\nu^+(K)-1$, then $\hfred(S_{p/q}^3(K),i)\cong \bigoplus_{s\in\Z} \ared{\lfloor\frac{i+ps}{q}\rfloor}$ as a $\Q$-graded groups, where the absolute grading on $\ared{\lfloor\frac{i+ps}{q}\rfloor}$ is determined by the absolute grading on the summand $(s,\Abold)\subset \Xipq$.
\end{prop}
\begin{rem}
When $q>1$,  the same group $\ared{k}$ can appear as a summand in $\hfred(S_{p/q}^3(K),i)$ for more than one value of $i$. Since the grading on $\Xipq$ depends on $i$, these summands will, in general, possess different gradings.
\end{rem}

The following lemma shows that under certain circumstances we can recover information about the knot Floer homology of two knots with the same surgery. It is a key technical ingredient in the proofs of Theorem~\ref{thm:HFKrecovery} and Theorem~\ref{thm:technical2}.

\begin{lem}\label{lem:HFKrecovery2}
Suppose that $S_{p/q}^3(K)\cong S_{p/q}^3(K')$, for some $p/q>2g(K)-1$, and that $\phi_{K,p/q}=\phi_{K',p/q}$ or $\phi_{K,p/q}=\overline{\phi_{K',p/q}}$. If either
\begin{enumerate}[(i)]
\item $S_{p/q}^3(K)$ is an $L$-space;
\item $q\geq 2$ and, for all $k$, there is $N_k\geq 0$ such that $\ared{k}(K)\cong \mathcal{T}(N_k)$; or
\item $q\geq 3$,
\end{enumerate}
then $V_k(K)=V_k(K')$ and $\ared{k}(K) \cong \ared{k}(K')$ for all $k\geq 0$.
\end{lem}
\begin{proof}
Since conjugation induces a grading preserving isomorphism on Heegaard Floer homology, the assumptions on $\phi_{K,p/q}$ and $\phi_{K',p/q}$, imply that we have an isomorphism
\[\hfp(S_{p/q}^3(K),i)\cong\hfp(S_{p/q}^3(K'),i)\]
as $\Q$-graded groups for all $0\leq i \leq p-1$. By comparing the $d$-invariants of these groups and applying \eqref{eq:NiWuformula}, this shows that $V_k(K)=V_k(K')$ for all $0\leq k \leq \lfloor \frac{p+q-1}{2q} \rfloor$. Since $p/q>2\nu^+(K)-1$, it follows that $V_{\lfloor \frac{p+q-1}{2q} \rfloor}(K)=V_{\lfloor \frac{p+q-1}{2q} \rfloor}(K')=0$. This shows that $V_k(K)=V_k(K')$ for all $k\geq 0$ and also that $\nu^+(K)=\nu^+(K')$. If $S_{p/q}^3(K)$ is an $L$-space, then we necessarily have $\ared{k}(K)= \ared{k}(K')=0$ for all $k$. Thus we can only need to establish the proposition under conditions $(ii)$ and $(iii)$. This is done by examining the absolute grading on the reduced part of $\hfp(S_{p/q}^3(K))$.

Since $p/q>2g(K)-1$, Proposition~\ref{prop:redHFp} shows that for any $0\leq i \leq \frac{p+q-1}{2}$ the reduced homology group takes the form
\[\hfred(S_{p/q}^3(K),i)\cong \ared{\lfloor \frac{i}{q} \rfloor}(K).\]
In particular, for $0\leq k<g(K)\leq \lfloor \frac{p+q-1}{2q} \rfloor$ we have
\[\hfred(S_{p/q}^3(K),kq)\cong \dotsb \cong \hfred(S_{p/q}^3(K),kq+q-1).\]
Moreover, by \eqref{eq:basegrading}, these isomorphisms preserve the absolute $\Q$-gradings up to a constant shift. That is, if there is an element of $\hfred(S_{p/q}^3(K),kq)$ with grading $x+d(p,q,kq)$, then for any $0\leq j\leq q-1$, there is an element of $\hfred(S_{p/q}^3(K),kq+j)$ with grading $x+d(p,q,kq+j)$. It is this constant grading shift property which we will use to prove the proposition.

Proposition~\ref{prop:redHFp} shows that for any $0\leq i \leq \frac{p+q-1}{2}$,
\[\hfred(S_{p/q}^3(K),i)\cong \ared{\lfloor \frac{i}{q} \rfloor}(K)\cong \ared{\lfloor \frac{i}{q} \rfloor}(K') \oplus \bigoplus_{s\ne 0} \ared{\lfloor\frac{i+ps}{q}\rfloor}(K').\]
Suppose that we do not have $\ared{k}(K) \cong \ared{k}(K')$ for all $k\geq 0$. Let $m\leq g(K)-1$ be maximal such that $\ared{m}(K) \not\cong \ared{m}(K')$.
This means that there is $s'\ne 0$ such that
\[\ared{\lfloor\frac{mq+q-1+ps'}{q}\rfloor}(K')\ne 0.\]
The maximality of $m$ implies $\ared{m+1}(K) \cong \ared{m+1}(K')$. It follows that we must have
\[\lfloor\frac{mq+q-1+ps'}{q}\rfloor < \lfloor\frac{mq+q+ps'}{q}\rfloor,\]
and hence that $q$ divides $ps'$. As $\gcd(p,q)=1$, this shows that $s'$ takes the form $s=tq$ for some $t\in\Z$.

For any $0\leq i \leq  \frac{p+q-1}{2}$, \eqref{eq:basegrading} and \eqref{eq:towergradshift} show that $(s,1)\in (s,\AT_{\lfloor\frac{ps+i}{q}\rfloor})\subseteq \Xipq$ has grading given by
\begin{equation}\label{eq:Abasegrading}
gr(s,1)=
\begin{cases}
d(p,q,i)- 2V_{\lfloor \frac{i+sp}{q} \rfloor}+ 2\sum_{k=1}^{s-1} \lfloor\frac{i+pk}{q}\rfloor &\text{if $s\geq 1$,}\\
d(p,q,i)- 2V_{\lfloor \frac{i+sp}{q} \rfloor}- 2\sum_{k=-s}^{0} \lfloor\frac{i+pk}{q}\rfloor &\text{if $s\leq 0$.}
\end{cases}
\end{equation}

If $\ared{m}(K)\cong \mathcal{T}(N_m)$, then it cannot be decomposed as a non-trivial direct sum of $\F[U]$-modules. So if $(ii)$ holds, then we must have
\[\hfred(S_{p/q}^3(K),mq)\cong \hfred(S_{p/q}^3(K),mq+1)\cong \ared{m+pt}(K'),\]
for some $t\ne 0$. However \eqref{eq:Abasegrading} shows that if $\ared{m+pt}(K')$ is to be endowed with the correct grading in both $\mathbb{X}^+_{mq,p/q}$ and $\mathbb{X}^+_{mq+1,p/q}$, then
\begin{align}\begin{split}\label{eq:gradingmatch}
\sum_{k=1}^{tq-1} \lfloor\frac{mq+pk}{q}\rfloor = \sum_{k=1}^{tq-1} \lfloor\frac{mq+1+pk}{q}\rfloor
 &\quad\text{if $t>0$ and} \\
\sum_{k=-tq}^{0} \lfloor\frac{mq+pk}{q}\rfloor = \sum_{k=-tq}^{0} \lfloor\frac{mq+1+pk}{q}\rfloor
&\quad\text{if $t<0$}.
\end{split}\end{align}
However, for any $r\in\Z$ such that $rp\equiv -1 \bmod{q}$, we have
\[\lfloor\frac{mq+pr}{q}\rfloor<\lfloor\frac{mq+1+pr}{q}\rfloor.\]
Since we can find such an $r$ in the range $1\leq r \leq q-1$, we see that the equalities in \eqref{eq:gradingmatch} cannot hold if $t\ne 0$. This completes the proof when condition $(ii)$ holds.

If $q\geq 3$, then consider any $t\ne 0$ for which $\ared{m+pt}(K')\ne 0$. By comparing the sums in \eqref{eq:Abasegrading} for different values of $i$, we see that if $t>0$ and $\ared{m+pt}(K')$ contributes a term with grading $x+d(p,q,mq+1)$ to $\hfred(S_{p/q}^3(K),mq+1)$, then the corresponding term it contributes to $\hfred(S_{p/q}^3(K),mq+2)$ has grading strictly greater than $x+d(p,q,mq+2)$. Similarly, if $t<0$ and $\ared{m+pt}(K')$ contributes a term with grading $x+d(p,q,mq+1)$ to $\hfred(S_{p/q}^3(K),mq+1)$, then the term it contributes to $\hfred(S_{p/q}^3(K),mq)$ has grading strictly greater than $x+d(p,q,mq)$. In particular, such an $\ared{m+pt}(K')$ always produces a grading on $\hfred(S_{p/q}^3(K),mq+1)$ which is too small when compared to the gradings on $\hfred(S_{p/q}^3(K),mq)$ and $\hfred(S_{p/q}^3(K),mq+2)$. This completes the proof when $q\geq 3$.
\end{proof}

\subsection{The $d$-invariants of lens spaces.}
In this section, we prove the congruence properties of $d$-invariants that we will require. Ozsv{\'a}th and Szab{\'o} have shown that the $d$-invariants of lens spaces can be calculated recursively
for $0\leq i\leq p-1$ using
\begin{equation}\label{eq:recursive1}
d(p,q,i)= -\frac{1}{4} + \frac{(p+q-1-2i)^2}{4pq} - d(q,r,i'),
\end{equation}
where $q\equiv r \bmod p$ and $i\equiv i' \bmod{q}$, and
\[d(1,0,0)=d(S^3)=0.\]

It will be temporarily convenient to work with a rescaled version of the $d$-invariants. Let $\tilde{d}(p,q,i)=2pd(p,q,i)$. By \eqref{eq:recursive1}, these satisfy
\begin{equation}\label{eq:recursivetilde}
\tilde{d}(p,q,i)=\frac{(p+q-1-2i)^2 -pq - 2p\tilde{d}(q,r,i')}{2q}.
\end{equation}

\begin{lem}\label{lem:2Zcondition}
For all $i,j$ in the range $0\leq i,j\leq p-1$, the quantity $\tilde{d}(p,q,i) -\tilde{d}(p,q,j)$ is an integer satisfying
\begin{equation}\label{eq:mod4cong}
\tilde{d}(p,q,i) -\tilde{d}(p,q,j)\equiv 2(i-j)(p+1) \bmod4
\end{equation}
and
\begin{equation}\label{eq:mod4pcong}
q(\tilde{d}(p,q,i) -\tilde{d}(p,q,j))\equiv 2(pq+q-1-i-j)(j-i) \mod{4p}
\end{equation}
\end{lem}
\begin{proof}
We prove both \eqref{eq:mod4cong} and \eqref{eq:mod4pcong} by induction on $p$. As $\tilde{d}(1,0,0)=0$, the required identities are clearly true for $p=1$. The inductive step is carried out by performing some elementary but slightly masochistic modular arithmetic.

From \eqref{eq:recursivetilde}, we have
\begin{equation}\label{eq:congproof}
q(\tilde{d}(p,q,i) -\tilde{d}(p,q,j)) = 2(p+q-1-(i+j))(j-i) +p(\tilde{d}(q,r,j')-\tilde{d}(q,r,i')),
\end{equation}
where $0\leq i',j',r\leq q-1$ are congruent modulo $q$ to $i,j$ and $p$ respectively. By the inductive hypothesis we know that
\begin{equation}\label{eq:congproof2}
\tilde{d}(q,r,j')-\tilde{d}(q,r,i')\equiv 2(j'-i')(q+1)\bmod{4}
\end{equation}
and
\begin{equation}\label{eq:congproof3}
r(\tilde{d}(q,r,j')-\tilde{d}(q,r,i'))\equiv 2(qr+r-1-i'-j')(i'-j') \mod{4q}.
\end{equation}
\paragraph{}We first prove \eqref{eq:mod4pcong} by reducing \eqref{eq:congproof} modulo $4p$. If $q$ is odd, then \eqref{eq:congproof2} shows that $\tilde{d}(q,r,j')-\tilde{d}(q,r,i')\equiv 0 \bmod 4$. Therefore, \eqref{eq:congproof} gives
\[q(\tilde{d}(p,q,i) -\tilde{d}(p,q,j))\equiv 2(p+q-1-i-j)(j-i) \mod 4p,\]
as required. If $q$ is even, then $i-j\equiv i'- j' \bmod 2$. So \eqref{eq:congproof2} shows that
\[p(\tilde{d}(q,r,j')-\tilde{d}(q,r,i'))\equiv 2(j-i)p \bmod 4p.\]
So when $q$ is even, \eqref{eq:congproof} gives
\[q(\tilde{d}(p,q,i) -\tilde{d}(p,q,j))\equiv 2(q-1-i-j)(j-i) \bmod 4p,\]
as required.

\paragraph{}To prove \eqref{eq:mod4cong}, we consider the result of reducing \eqref{eq:congproof} modulo $4q$. If we write $i=i'+\alpha q$ and $j=j' + \beta q$, then one can check that
\begin{equation}\label{eq:congproof4}
2(p+q-1-i-j)(j-i)\equiv 2(p+q-1-i'-j')(j'-i') + 2q(\alpha + \beta)(p+1) \bmod{4q}.
\end{equation}
On the other hand, by using \eqref{eq:congproof2} and \eqref{eq:congproof3}, we find that
\begin{align}\begin{split}\label{eq:congproof5}
p(\tilde{d}(q,r,j')-\tilde{d}(q,r,i'))&\equiv 2(p-r)(i'-j')(q+1) + 2(qr + r-1-i'-j')(i'-j') \bmod{4q}\\
&\equiv 2(i'-j')(pq+p-1-i'-j')\bmod{4q}.
\end{split}\end{align}
By summing \eqref{eq:congproof4} and \eqref{eq:congproof5}, we obtain
\begin{align}\begin{split}\label{eq:congproof6}
q(\tilde{d}(p,q,i) -\tilde{d}(p,q,j)) &\equiv 2q(p+1)(i'-j') + 2q(\alpha + \beta)(p+1) \bmod{4q}\\
&\equiv 2q(p+1)(i'+j'+\alpha +\beta) \bmod{4q}.
\end{split}\end{align}
Since the right hand side of \eqref{eq:congproof6} is divisible by $q$, it follows that $\tilde{d}(p,q,i) -\tilde{d}(p,q,j)$ is an integer. If $p$ is odd, then $p+1$ is even, so \eqref{eq:congproof6} shows that
\[q(\tilde{d}(p,q,i) -\tilde{d}(p,q,j))\equiv 0 \bmod{4q}.\]
Thus if $p$ is odd, then $\tilde{d}(p,q,i) -\tilde{d}(p,q,j)\equiv 0 \bmod{4}$, as required. If $p$ is even, then $q$ is necessarily odd. In this case we have $i + j \equiv i'+j'+\alpha+\beta\bmod{2}$. So \eqref{eq:congproof6} also implies \eqref{eq:mod4cong} when $q$ is odd. This completes the proof.
\end{proof}

This allows us to prove the congruence result for $d$-invariants we require.
\begin{cor}\label{cor:integralconditions}
We have
\[d(p,q,i) -d(p,q,j)\in 2\Z \Leftrightarrow
\begin{cases}
(q-1-i-j)(j-i) \equiv 0 \bmod{p} &\text{ if $p$ is odd}\\
(q-1-i-j)(j-i) \equiv 0 \bmod{2p} &\text{if $p$ is even}.
\end{cases}
\]
\end{cor}
\begin{proof}
We prove this by showing that
\[\tilde{d}(p,q,i) -\tilde{d}(p,q,j)\in 2p\Z \Leftrightarrow
\begin{cases}
(q-1-i-j)(j-i) \equiv 0 \bmod{p} &\text{ if $p$ is odd}\\
(q-1-i-j)(j-i) \equiv 0 \bmod{2p} &\text{if $p$ is even}.
\end{cases}\]

\paragraph{}If $p$ is odd, then Lemma~\ref{lem:2Zcondition} shows
\[\tilde{d}(p,q,i) -\tilde{d}(p,q,j) \equiv 0 \bmod{4}.\]
Consequently we see that
$\tilde{d}(p,q,i) -\tilde{d}(p,q,j)\equiv 0 \bmod{4p}$ if and only if
$\tilde{d}(p,q,i) -\tilde{d}(p,q,j)\equiv 0 \bmod{p}$. From \eqref{eq:mod4pcong}, we see that
\[q(\tilde{d}(p,q,i) -\tilde{d}(p,q,j))\equiv 2(q-1-i-j)(j-i) \bmod{p}.\]
Since $q$ is coprime to $p$, this is congruent to 0 if and only if
\[(q-1-i-j)(j-i) \equiv 0 \bmod{p},\]
as required.

\paragraph{}If $p$ is even, then $q$ is necessarily odd. As an odd $q$ is invertible modulo $4p$, it follows from \eqref{eq:mod4pcong} that $\tilde{d}(p,q,i) -\tilde{d}(p,q,j)\equiv 0 \bmod{4p}$ if and only if $2(q-1-i-j)(j-i)\equiv 0 \bmod 4p$. Equivalently, if and only if
\[(q-1-i-j)(j-i)\equiv 0 \bmod 2p,\] as required.
\end{proof}

One other result we will require is an a bound on the absolute value of the $d$-invariants of lens spaces.
\begin{lem}\label{lem:overalldbound}
For any $1\leq q \leq p-1$ and any $0\leq i\leq p-1$, we have
\[|d(p,q,i)|\leq \frac{p-1}{4}.\]
\end{lem}
\begin{proof}
Since the $d$-invariants of a rational homology sphere satisfy $d(-Y,\spincs)=-d(Y,\spincs)$ for any $\spincs\in\spinc(Y)$, we see that for any $1\leq i\leq p-1$, there is $0\leq j\leq p-1$ such that $d(p,q,i)=-d(p,p-q,j)$. Therefore, to prove the lemma, it is sufficient to show that $d(p,q,i)\geq\frac{1-p}{4}$.
Since $d(1,0,0)=0$, we can assume $p>1$. Suppose that $p/q>1$ has a continued fraction expansion
\[p/q= a_1 - \cfrac{1}{a_2
           - \cfrac{1}{\ddots
           - \cfrac{1}{a_l} } },\]
where $a_i\geq 2$ for all $i$. This expansion has length $l\leq p-1$. If $M$ is the matrix
\[M =
  \begin{pmatrix}
   a_1  & -1   &        &       \\
   -1   & a_2  & -1     &       \\
        & -1   & \ddots & -1    \\
        &      & -1     & a_l
  \end{pmatrix},\]
then Ozsv{\'a}th and Szab{\'o} have shown that for any $i$, there is $v\in \Z^l$ such that \cite{Ozsvath03plumbed,ozsvath2005heegaard}
\[4d(p,q,i)= v^T M^{-1} v - l.\]
Since $M$ and, hence $M^{-1}$, is positive definite, this shows that
\[4d(p,q,i)\geq l \geq 1-p.\]
This gives the desired lower bound.
\end{proof}

\section{Proving Theorems~\ref{thm:HFKrecovery} and~\ref{thm:technical2}}
Let $Y$ be a 3-manifold such that $Y\cong S_{p/q}^3(K') \cong S_{p/q}^3(K)$ for knots $K$ and $K'$ in $S^3$ and some $p/q>0$. By \eqref{eq:NiWuformula}, these two surgery descriptions of $Y$ give labelings
\begin{equation*}
\phi_{K,p/q},\phi_{K',p/q} \colon \Zp \longrightarrow \spinc(Y),
\end{equation*}
such that the following diagram commutes
 \[\xymatrix{
 {\rm Spin}^c(Y) \ar[rd]^d  & \mathbb{Z}/p\mathbb{Z} \ar[l]_{\phi_{K,p/q}} \ar[d]^{D} \\
 \mathbb{Z}/p\mathbb{Z} \ar[u]^{\phi_{K',p/q}}\ar[r]^{D'} & \mathbb{Q},
}\]
where
\[
D(i)=d(p,q,i)-2\max\{V_{\lfloor\frac{i}{q}\rfloor}(K),V_{\lceil\frac{p-i}{q}\rceil}(K)\}
\]
and
\[
D'(i)=d(p,q,i)-2\max\{V_{\lfloor\frac{i}{q}\rfloor}(K'),V_{\lceil\frac{p-i}{q}\rceil}(K')\}
\]
for $0\leq i \leq p-1$.

Thus if we let $\phi$ denote map
\[\phi :=\phi_{K,p/q}^{-1}\circ \phi_{K',p/q}\colon \Zp \rightarrow \Zp\]
and $f(i)=\min\{\lfloor \frac{i}{q} \rfloor,\lceil \frac{p-i}{q} \rceil\}$,
then for $0\leq i \leq p-1$, we have
\begin{equation}\label{eq:Vrelation}
2V_{f(i)}(K)-2V_{f(\phi(i))}(K')=d(p,q,i)-d(p,q,\phi(i)).
\end{equation}
\begin{rem}
There are two important consequences of \eqref{eq:Vrelation}. Firstly, it shows
\[d(p,q,i)-d(p,q,\phi(i))\in 2\Z.\]
Secondly, $d(p,q,i)-d(p,q,\phi(i))>0$ implies that $V_{f(i)}(K)>0$ and hence that $\nu^+(K)\geq f(i)+1$.
\end{rem}
There are three possible forms for $\phi$.
\begin{prop}\label{prop:mapform}
The map $\phi \colon \Zp \rightarrow \Zp$ takes one of the  following forms:
\begin{enumerate}[I:]
\item $\phi(i)=a(i-s)+s \bmod p$, where $p$ is odd, $a^2 \equiv 1 \bmod p$ and $s\in \{\frac{q-1}{2},\frac{p+q-1}{2}\}\cap \Z$.
\item $\phi(i)=a(i-s)+s \bmod p$, where $p$ is even, $a^2\equiv 1 \bmod 2p$ and $s=\frac{q-1}{2}$; or
\item $\phi(i)=a(i-s)+s+\frac{p}{2} \bmod p$, where $p\equiv 0 \bmod 8$, $a^2\equiv p+1 \bmod 2p$ and $s=\frac{q-1}{2}$.
\end{enumerate}
\end{prop}
\begin{proof}
Let $J:\Zp \rightarrow \Zp$ be the map $J(i)=q-1-i \bmod p$. Since the $d$-invariants are invariant under conjugation, \eqref{eq:conjugation} shows that $J\circ \phi = \phi \circ J$.
Since $\phi$ is an affine bijection, we may assume that it can be written in the form
\[\phi(i)=a(i-s_0)+s_1 \bmod{p},\]
for some $s_0 \in {\rm Fix}(J)=\{\frac{q-1}{2},\frac{p+q-1}{2}\}\cap \Z$ and some $a\in (\Zp)^\times$. Using the invariance of $d$-invariants under conjugation, we obtain
\[J(s_1)=J(\phi(s_0))=\phi(J(s_0))=s_1.\]
This shows that we also have $s_1 \in {\rm Fix}(J)$.

\paragraph{}First assume that $p$ is odd. Since $|{\rm Fix}(J)|=1$ in this case, we have $s_0=s_1$.
Since $d(p,q,1+s_0)-d(p,q,\phi(1+s_0))\in 2\Z$, Corollary~\ref{cor:integralconditions} shows that
\begin{align*}
0 &\equiv (\phi(1+s_0)+1+s_0-q+1)(\phi(1+s_0)-(1+s_0)) \bmod p\\
 &\equiv (a + s_0 + 1+s_0 - (q-1))(a +s_0-1-s_0))\bmod p \\
  &\equiv a^2-1 \bmod p.
\end{align*}
This shows $\phi$ takes the form given by type I.

\paragraph{}Now assume that $p$ is even and $s_0=s_1$. Since $q$ is necessarily odd, we may assume that $s_0=\frac{q-1}{2}$. Since $d(p,q,1+s_0)-d(p,q,\phi(1+s_0))\in 2\Z$, Corollary~\ref{cor:integralconditions} shows that
\begin{align*}
0 &\equiv (\phi(1+s_0)+1+s_0-q+1)(\phi(1+s_0)-(1+s_0)) \bmod 2p\\
 &\equiv (a + s_0 + 1+s_0 - (q-1))(a +s_0-1-s_0))\bmod 2p \\
  &\equiv a^2-1 \bmod 2p.
\end{align*}
This shows $\phi$ takes the form given by type II.

\paragraph{} Finally, assume that $p$ is even and $s_0\ne s_1$. We may assume that $s_1=\frac{p+q-1}{2}=s_0+\frac{p}{2}$. Since $d(p,q,s_0)-d(p,q,\phi(s_0))\in 2\Z$ Corollary~\ref{cor:integralconditions} shows that
\begin{align*}
0 &\equiv (\phi(s_0)+s_0-q+1)(\phi(s_0)-s_0) \bmod 2p\\
 &\equiv (s_0+s_1 - (q-1))(s_1-s_0)\bmod 2p \\
 &\equiv \frac{p^2}{4} \bmod 2p,
\end{align*}
which implies that $p\equiv 0 \bmod 8$. Similarly, from $d(p,q,s_0+1)-d(p,q,\phi(s_0+1))\in 2\Z$, we obtain
\begin{align*}
0 &\equiv (\phi(1+s_0)+1+s_0-q+1)(\phi(1+s_0)-1-s_0) \bmod 2p\\
 &\equiv (a+s_1 + 1+s_0 - (q-1))(a+s_1-s_0-1)\bmod 2p \\
 &\equiv (a+\frac{p}{2})^2-1 \bmod 2p\\
 &\equiv a^2 +p -1 \bmod 2p.
\end{align*}
This shows $\phi$ takes the form given by type III.
\end{proof}

This allows us to put bounds on $\nu^+(K)$ when $\phi$ is not the identity or the map corresponding to conjugation.
\begin{lem}\label{lem:genusbound}
If $\phi_{K,p/q}\ne\phi_{K',p/q}$ and $\phi_{K,p/q}\ne\overline{\phi_{K',p/q}}$, then
\[\nu^+(K)> \frac{p}{4q}+\frac{1}{2} - \frac{3}{q}-q.\]
\end{lem}
\begin{proof}
Consider the map
\[\phi =\phi_{K,p/q}^{-1}\circ \phi_{K',p/q}\colon \Zp \rightarrow \Zp\]
Proposition~\ref{prop:mapform} shows that if $\phi_{K,p/q}\ne\phi_{K',p/q}$ and $\phi_{K,p/q}\ne\overline{\phi_{K',p/q}}$, then $\phi$ takes the form $\phi(x)=a(x-s_0)+s_1$ for some $a\not\equiv \pm 1 \bmod{p}$ satisfying $a^2\equiv 1 \bmod{p}$. Since $d$-invariants are invariant under conjugation, we can assume that $a$ lies in the range $\sqrt{p}< a <p/2$.

\paragraph{} Since $\phi$ satisfies $\phi(x+n)\equiv \phi(x)+na \bmod{p}$ for all $x$ and $n$, we see that for any $\frac{p+q-1}{2}\geq N\geq \frac{p}{a}$, we can find $x$ in the range $N-\frac{p}{a}\leq x\leq N$ such that
\[\frac{p+q-1-a}{2}\leq \phi(x)\leq \frac{p+q-1+a}{2}.\]
For such an $x$ we have
\begin{align*}
d(p,q,x)-d(p,q,\phi(x))&= \frac{(p+q-1-x-\phi(x))(\phi(x)-x)}{pq}-d(q,r,x) +d(q,r,\phi(x))\\
&> \frac{(p+q-1+a-2x)(p+q-1-a-2x)}{4pq} - \frac{q}{2}\\
&= \frac{(p+q-1-2x)^2-a^2}{4pq} - \frac{q}{2}\\
&\geq \frac{(p+q-1-2N)^2-a^2}{4pq} - \frac{q}{2},
\end{align*}
where we used the bound $|d(q,r,\phi(x))-d(q,r,x)|<\frac{q}{2}$ arising from Lemma~\ref{lem:overalldbound} to obtain the second line, and the final line was obtained by observing that the quadratic in the preceding line is minimized for $x$ in the range $N-\frac{p}{a}\leq x\leq N$ by taking $x=N$.

\paragraph{}Thus if we take
\[N=\frac{p+q-1 - \sqrt{2pq^2+a^2}}{2},\]
then such an $x$ satisfies $d(p,q,x)-d(p,q,\phi(x))>0$. And hence we see that there is $k\in \Z$ such that $V_k>0$ and
\begin{equation}\label{eq:kbound}
k\geq \lfloor\frac{p+q-1 - \sqrt{2pq^2+a^2}}{2q}-\frac{p}{aq}\rfloor.
\end{equation}
We complete the proof by finding a lower bound for $k$ which is independent of $a$.

Consider $\frac{\sqrt{2pq^2+a^2}}{2q}+\frac{p}{aq}$ as a function of $a$. For $a>0$, this has a single critical value which is a minimum. Thus we see that for $a$ in the range $\sqrt{p}\leq a \leq \frac{p}{2}$ the minimal value of the right hand side of \eqref{eq:kbound} is attained by $a=\sqrt{p}$ or $a=\frac{p}{2}$. Therefore, using the bound $\sqrt{2pq^2+a^2}\leq  a+\frac{pq^2}{a}$ when $a=\frac{p}{2}$, we obtain
\[k\geq \lfloor\frac{p+q-1}{2q} - \max\{\frac{p+4q^2+8}{4q},\frac{\sqrt{p}(2+\sqrt{2q^2+1})}{2q}\}\rfloor.\]
However, one can show that\footnote{When considered as a quadratic in $\sqrt{p}$, the discriminant of
\[p-2(2+\sqrt{2q^2+1})\sqrt{p} + 4q^2+10\]
is
\[\Delta=4\sqrt{2q^2+1}-2q^2-4.\]
As this satisfies $\Delta<0$ for all $q$, we see that
\[\frac{\sqrt{p}(2+\sqrt{2q^2+1})}{2q}<\frac{p+4q^2+10}{4q}\]
for all $p$ and $q$.}
\[\max\{\frac{p+4q^2+8}{4q},\frac{\sqrt{p}(2+\sqrt{2q^2+1})}{2q}\}< \frac{p+4q^2+10}{4q}.\]
This gives the bound
\[k\geq \lfloor\frac{p}{4q}+\frac{1}{2} - \frac{3}{q}-q\rfloor,\]
showing that
\[\nu^+(K)> \frac{p}{4q}+\frac{1}{2} - \frac{3}{q}-q,\]
as required.
\end{proof}

We now have all the pieces to prove or main technical results.
{\customthm{Theorem~\ref{thm:HFKrecovery}}
Let $K,K'\subseteq S^3$ be knots such that $S_{p/q}^3(K)\cong S_{p/q}^3(K')$. If
\[|p| \geq 12+4q^2-2q +4qg(K)
\quad \text{and} \quad q\geq 3,\]
then $\Delta_K(t)=\Delta_{K'}(t)$, $g(K)=g(K')$ and $K$ is fibred if and only if $K'$ is fibred.}
\begin{proof}
Suppose that $S_{p/q}^3(K)\cong S_{p/q}^3(K')$. Since $S_{-p/q}^3(K)\cong -S_{-p/q}^3(\overline{K})$ for any $K \subseteq S^3$, we may assume that $p/q>0$. By Lemma~\ref{lem:genusbound}, the bound
\[p \geq 12+4q^2-2q +4qg(K)\]
implies that either $\phi_{K,p/q}=\phi_{K',p/q}$ or $\phi_{K,p/q}=\overline{\phi_{K',p/q}}$. In either case, the assumption $q\geq 3$ allows us to apply Lemma~\ref{lem:HFKrecovery2}, which shows that $V_k(K)=V_k(K')$ and $\ared{k}(K)\cong \ared{k}(K')$ for all $k\geq 0$. By Proposition~\ref{prop:aredproperties}, this is sufficient information to guarantee that $K$ and $K'$ have the same Alexander polynomial, genera and fibredness.
\end{proof}

{\customthm{Theorem~\ref{thm:technical2}}
Suppose that $K$ is an $L$-space knot. If $S_{p/q}^3(K)\cong S_{p/q}^3(K')$ for some $K'\subseteq S^3$ and either
\begin{enumerate}[(i)]
\item $p\geq 12+4q^2 -2q +4qg(K)$ or
\item $p\leq \min\{2q-12-4q^2, -2qg(K)\}$ and $q\geq 2$
\end{enumerate}
holds, then $\Delta_K(t)=\Delta_{K'}(t)$, $g(K)=g(K')$ and $K'$ is fibred.}
\begin{proof}
Suppose that $S_{p/q}^3(K)\cong S_{p/q}^3(K')$, where $K$ is an $L$-space knot. First suppose that
\[p\geq 4q^2+12 -2q + 4qg(K).\] In this case, $S_{p/q}^3(K)$ is an $L$-space. Lemma~\ref{lem:genusbound} implies that either $\phi_{K,p/q}=\phi_{K',p/q}$ or $\phi_{K,p/q}=\overline{\phi_{K',p/q}}$. Thus we can apply Lemma~\ref{lem:HFKrecovery2} which shows that $V_k(K)=V_k(K')$ and $\ared{k}(K)\cong \ared{k}(K')=0$ for all $k$. By Proposition~\ref{prop:aredproperties}, this is sufficient information to guarantee that $K$ and $K'$ have the same Alexander polynomial and genus. As $K$ is fibred it also shows that $K'$ is fibred.

If $p\leq \min\{2q-12-4q^2, -2qg(K)\}$ and $q\geq 2$, then we use the fact that $S_{-p/q}^3(\overline{K})\cong -S_{p/q}^3(\overline{K'})$. As $\nu^+(\overline{K})=0$, the condition $-p\geq 4q^2+12-2q$ shows that $\phi_{\overline{K},-p/q}=\phi_{\overline{K'},-p/q}$ or $\phi_{\overline{K},-p/q}=\overline{\phi_{\overline{K'},-p/q}}$. As $-p/q>2qg(K)-1$ and $\ared{k}(\overline{K})\cong \mathcal{T}(V_{|k|}(K))$ for all $k$, Lemma~\ref{lem:HFKrecovery2} shows that $V_k(\overline{K})=V_k(\overline{K'})=0$ and $\ared{k}(\overline{K})\cong \ared{k}(\overline{K'})$ for all $k\geq 0$. This shows that $K$ and $K'$ have the same Alexander polynomial and genus. As $K$ is fibred it also shows that $K'$ is fibred.
\end{proof}

\section{Surgeries on torus knots}
In this section, we prove Theorem~\ref{thm:toruscharslopes}. In order to do this we will need to understand the manifolds obtained by surgery on torus knots. It is well-known that, with the exception of a single reducible surgery for each torus knot, the manifolds obtained by surgery on a torus knot are Seifert fibred spaces \cite{Moser71elementary}. We will use $S^2(e; \frac{b_1}{a_1}, \frac{b_2}{a_2},\frac{b_3}{a_3})$ to denote the 3-manifold obtained by surgery on the link given in Figure~\ref{fig:sfspace}.\footnote{In this notation we have
\[S_1^3(T_{3,2})\cong S^2(-2; \frac{1}{2},\frac{2}{3},\frac{4}{5})\cong P,\]
 where $P$ is the Poincar\'{e} sphere oriented so that it bounds the positive-definite $E_8$-plumbing.}
This is a Seifert fibred space when $a_i\ne 0$ for $i=1,2,3$ and is a lens space only if $|a_i|=1$ for some $i$. Recall that if $|a_i|>1$ for $i=1,2,3$, then
\[S^2(e; \frac{b_1}{a_1}, \frac{b_2}{a_2},\frac{b_3}{a_3}) \cong S^2(e'; \frac{d_1}{c_1}, \frac{d_2}{c_2},\frac{d_3}{c_3})\]
if and only if
\[e+ \frac{b_1}{a_1}+ \frac{b_2}{a_2}+\frac{b_3}{a_3}= e'+ \frac{d_1}{c_1}+ \frac{d_2}{c_2}+\frac{d_3}{c_3}\]
and there is a permutation $\pi$ of $\{1,2,3\}$ such that
\[\frac{b_i}{a_i}\equiv \frac{c_{\pi(i)}}{d_{\pi(i)}} \bmod 1 \quad\text{for $i=1,2,3$ \cite{neumann78seifert}.}\]
\begin{figure}
\centering
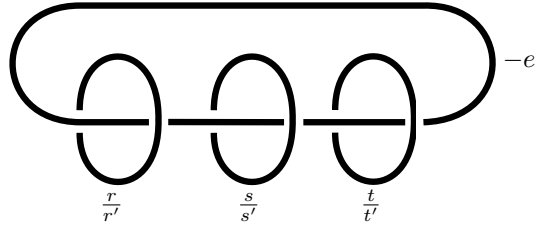
\caption{A surgery presentation for the Seifert fibred space $S^2(e; \frac{b_1}{a_1}, \frac{b_2}{a_2},\frac{b_3}{a_3})$}\label{fig:sfspace}
\end{figure}
We can describe surgeries on torus knots as follows.
\begin{prop}[cf. Moser \cite{Moser71elementary}]\label{prop:torussurgerycalculation}
For $r,s>1$ and any $p/q$,
\[
S_{p/q}^3(T_{r,s})\cong S^2(e; \frac{s'}{s},\frac{r'}{r},\frac{q}{p-rsq}),
\]
where $r'$, $s'$ and $e$ are any integers satisfying $rs'+sr'+ers=1$.
\end{prop}
\begin{proof}
Consider the Seifert fibration of $S^3$ with exceptional fibres of order $r$ and $s$ for which the regular fibres are isotopic to the torus knot $T_{r,s}$. We can obtain $S_{p/q}^3(T_{r,s})$ by surgering a regular fibre $K\subseteq S^3$. Let $\mu$ be a meridian for $K$ and $\lambda$ a null-homologous longitude. As the linking number between $K$ and a nearby regular fibre is $rs$, the surgery slope can be written as
 \[p\mu + q \lambda = (p-qrs)\mu + q\kappa,\]
where $\kappa$ is a longitude for $K$ given by a regular fibre. This shows that
 \[S_{p/q}^3(T_{r,s})\cong S^2(e; \frac{s'}{s},\frac{r'}{r},\frac{q}{p-rsq}),\]
 for some $s',r',e\in \Z$ which are independent of $p/q$. Considering the order of the homology group $H_1(S_{p/q}^3(T_{r,s}))$ shows that
 \begin{align*}
 |p|&=|rs(p-rsq)(e+\frac{s'}{s}+\frac{r'}{r}+\frac{q}{p-rsq})|\\
 &=|p(rs'+sr'+rse)-qrs(rs'+sr'+ers-1)|.
 \end{align*}
 As this holds for any $q$ it follows that we have $rs'+sr'+ers=1$, as required.
\end{proof}

We will also use the Casson-Walker invariant \cite{walker92extension}. For any rational homology sphere $Y$, this is a rational-valued invariant $\lambda(Y)\in\Q$. For our purposes its most useful property is that it is easily computed for manifolds obtained by surgery. For any knot $K'\subseteq S^3$, the Casson-Walker invariant satisfies \cite{Boyer90cassonsurgery}:
\begin{equation}\label{eq:BoyerLines}
\lambda(S_{p/q}^3(K'))=\lambda(S^3_{p/q}(K)) + \frac{q}{2p} \Delta_{K'}''(1).
\end{equation}
In particular, this means that if $S_{p/q}^3(K)\cong S_{p/q}^3(T_{r,s})$, then the Alexander polynomial of $K$ satisfies
\begin{equation}\label{eq:secondderiv}
\Delta_{K}''(1)=\Delta_{T_{r,s}}''(1)=\frac{(r^2-1)(s^2-1)}{12}.
\end{equation}
Recall also that if $K$ is a satellite knot with pattern $P$ and companion $K'$, then its Alexander polynomial satisfies
\begin{equation*}\label{eq:satellitealexpoly}
\Delta_{K}(t)=\Delta_{K'}(t^w)\Delta_{P}(t),
\end{equation*}
where $w$ is the winding number of $P$. In particular, this means that
\begin{equation}\label{eq:satellitederiv}
\Delta_{K}''(1)= \Delta_{P}''(1)+ w^2\Delta_{K'}''(1).
\end{equation}
Next we consider the possibility that a torus knot and a cable of a torus knot share a non-integer surgery.

{\customthm {Proposition~\ref{prop:cablenoncharslopes}}
For $s>r>1$ and $q\geq 2$, there exists a non-trivial cable of a torus knot $K$ such that $S_{p/q}^3(K)\cong S_{p/q}^3(T_{r,s})$ if and only if
\[s=\frac{rq^3 \pm 1}{q^2-1},\quad p= \frac{r^2 q^4 -1}{q^2-1}, \quad q=\lfloor s/r\rfloor \quad\text{and}\quad r>q,\]
in which case $K$ is the $(q,\frac{q^2r^2-1}{q^2-1})$-cable of $T_{r,\frac{rq\pm 1}{q^2-1}}$.}

\begin{proof}
Suppose that $K$ is the $(w,c)$-cable of $T_{a,b}$, where $w>1$ is the winding number of the pattern and that this satisfies
\[Y\cong S_{p/q}^3(K)\cong S_{p/q}^3(T_{r,s}).\]
If $Y$ is reducible, then $p/q=wc=rs$ is an integer. So we can assume from now on that $Y$ is irreducible. We will temporarily drop the assumption that $s>r$ and assume for now only that $r,s>1$.

Since $Y$ does not contain an incompressible torus, $p/q$ must take the form
\[\frac{p}{q}=wc\pm \frac{1}{q}\]
and $Y\cong S_{p/(qw^2)}^3(T_{a,b})$ \cite{Gordon83Satellite}. As $q>1$ and $\gcd(w,c)=1$, we have
\[|qwc \pm 1 -abqw^2|=|qw(c-abw)\pm 1|\geq |qw|-1.\]
This shows that $Y$ is not a lens space if $q>1$. Thus we can assume $Y$ is a Seifert fibred space with base orbifold
\[S^2(r,s,|p-rsq|) = S^2(|a|,|b|,|p-abqw^2|).\]
By considering the order of the exceptional fibers, we see we may assume that $a=r$. Using the Casson-Walker invariant as in \eqref{eq:secondderiv} and \eqref{eq:satellitederiv}, we see that
\[(a^2-1)(s^2-1)= w^2(a^2-1)(b^2-1)+ (c^2-1)(w^2-1).\]
This implies that $|b|<s$. Thus we must have
\[\varepsilon_1 s = p-abqw^2
\quad \text{and} \quad
\varepsilon_2 b = p-asq\]
for some $\varepsilon_1,\varepsilon_2 \in \{\pm1\}$. Solving these simultaneous equations shows that
\begin{equation}\label{eq:bsformulae}
s=\frac{p(aqw^2-\varepsilon_2)}{(aqw)^2-\varepsilon_1 \varepsilon_2}
\quad \text{and}\quad
b=\frac{p(aq-\varepsilon_1)}{(aqw)^2-\varepsilon_1 \varepsilon_2}.
\end{equation}
Since $a,s>1$ it follows that $p>0$ and $b>0$.

Using Proposition~\ref{prop:torussurgerycalculation} and our two surgery descriptions, we see that $Y$ can be written in the form
\begin{equation*}
Y\cong S^2(-1;\frac{a'}{a},\frac{s'}{s},\frac{\varepsilon_2q}{b})\cong S^2(-1;\frac{a'}{a},\frac{\varepsilon_1qw^2}{s},\frac{b'}{b})
\end{equation*}
for some $a',b',s'\in \Z$, where $1\leq b'<b$ and $1\leq s' <s$. Comparing these two descriptions of $Y$ as a Seifert fibred space we see that $b'$ satisfies
\begin{equation}\label{eq:bproperties}
b'\equiv \varepsilon_2q \bmod{b}
\quad \text{and} \quad
b'a \equiv 1 \bmod{b};
\end{equation}
$s'$ satisfies
\begin{equation}\label{eq:sproperties}
s'\equiv \varepsilon_1qw^2 \bmod{s}
\quad \text{and} \quad
s'a \equiv 1 \bmod{s};
\end{equation}
and we have
\begin{equation}\label{eq:sumseifinv}
-1+ \frac{a'}{a}+ \frac{\varepsilon_2q}{b}+\frac{s'}{s}=-1+ \frac{a'}{a}+\frac{\varepsilon_1qw^2}{s}+\frac{b'}{b}.
\end{equation}
\begin{claim}
We have $\varepsilon_1=\varepsilon_2$.
\end{claim}
\begin{proof}[Proof of Claim]
By \eqref{eq:sumseifinv}, we see that
\[\frac{s'-\varepsilon_1qw^2}{s}=\frac{b'-\varepsilon_2q}{b}.\]
Equations \eqref{eq:bproperties} and \eqref{eq:sproperties} show that both sides of this equation are integers. Moreover the assumptions that $1\leq s'<s$ and $1\leq b'<b$ imply that the right hand side (respectively left hand side) is strictly greater than 0 if and only if $\varepsilon_1=-1$ (respectively $\varepsilon_2=-1$). It follows that $\varepsilon_1=\varepsilon_2$, as required.
\end{proof}
 In light this claim,  we will take $\varepsilon= \varepsilon_1= \varepsilon_2\in \{\pm1\}$ from now on.

Observe that \eqref{eq:bproperties} and \eqref{eq:sproperties} imply that $b$ divides $aq-\varepsilon$ and $s$ divides $aqw^2-\varepsilon$. Combining this with \eqref{eq:bsformulae} shows that there is a positive integer $k$ such that
\[kb=aq-\varepsilon,
\quad
ks=aqw^2-\varepsilon
\quad \text{and} \quad
kp= (aqw)^2-1.
\]
Since we can write
\[aqw^2-\varepsilon=w^2(aq-\varepsilon)+\varepsilon(w^2-1),\]
we also see that $k$ also divides $w^2-1$.

Now $p$ takes the form $p=qwc + \delta$ for some $\delta\in \{\pm 1\}$. Therefore we have
\[kp=qwck + \delta k = (qwa)^2 - 1.\]
In particular, we have
\begin{equation}\label{eq:deltakformula}
\delta k = qwN - 1,
\quad \text{where}\quad
N=qwa^2 - ck.
\end{equation}

\begin{claim}
We have $qN\in \{0,1,w\}$.
\end{claim}
\begin{proof}[Proof of Claim]
Assume that $qN$ is non-zero. Since $k$ divides $w^2-1$, \eqref{eq:deltakformula} shows that $qwN - 1$ divides $w^2-1$. Let $\alpha \in\Z$ be such that $(qwN-1)\alpha=w^2-1$. By considering this equation mod $w$, we see that $\alpha$ takes the form $\alpha=\beta w+1$ for some $\beta\in \Z$. Substituting for $\alpha$ and  rearranging shows that $\beta$ satisfies
\[(\beta qN-1)w=\beta-qN.\]
If $\beta \leq -1$ or $\beta\geq 2$, then there are no possible integer choices of $qN$ for which $w=\frac{\beta -qN}{\beta qN-1}\geq 2$ holds. So it remains only to consider are $\beta=0$ and $\beta=1$. The former implies that $qN=w$ and the latter can only hold if $qN=1$, completing the proof of the claim.
\end{proof}
This gives us three 3 possibilities for $qN$ to consider. If $qN=0$, then $N=0$ and \eqref{eq:deltakformula} shows that $\delta k=-1$. This implies that $k=1$ and hence $c=qwa^2$. This contradict the fact that $\gcd(w,c)=1$.

As we are assuming $q\geq 2$, $qN=1$ can't possibly hold. Thus it remains only to consider the possibility $qN=w$. In this case \eqref{eq:deltakformula} shows that
\[\delta=1, \quad k=w^2-1\, \quad \text{and,}\quad c=\frac{w(q^2a^2-1)}{q(w^2-1)}.\]
As $\gcd(w,c)=1$, $\gcd(q^2a^2-1,q)=1$ and $\gcd(w^2-1,w)=1$, this shows that $q=w$.

It follows that $s=\frac{rq^3 -\varepsilon}{q^2-1}$, $p= \frac{r^2 q^4 -1}{q^2-1}$, $c=\frac{q^2r^2-1}{q^2-1}$ and $b=\frac{rq-\varepsilon}{q^2-1}$. This shows that $K$ must take the required form and that $s$ and $p$ satisfy the required conditions.

As $K$ is a non-trivial cable we have $b=\frac{rq-\varepsilon}{q^2-1}>1$. This will allow us to derive the condition that $r>q$. If $\varepsilon=1$, then $b>1$ clearly implies $r>q$. If $\varepsilon=-1$, then $b\in \Z$ implies that $r\equiv -q \bmod{q^2-1}$. Thus we have $r\geq q^2-q-1$, which implies that $b\geq q-1$ with equality only if $r= q^2-q-1$. Thus $b>1$ implies that either $q>2$ or $r\geq 2q^2-q-2$. In either case this is sufficient to guarantee that $r>q$ when $\varepsilon=-1$.

Finally, observe that
\[rq<s=rq+\frac{rq-\varepsilon}{q^2-1}\leq rq+ \frac{2r+1}{3}<r(q+1),\]
which implies that $q=\lfloor s/r \rfloor$. This completes the proof of one direction.

Conversely, if $s=\frac{rq^3 \pm 1}{q^2-1}\in \Z$, then $r\equiv \mp q \bmod{q^2-1}$, so we have $p=\frac{r^2 q^4 -1}{q^2-1}\in \Z$, $b=\frac{rq \pm 1}{q^2-1} \in\Z$ and $c=\frac{q^2r^2-1}{q^2-1}\in \Z$. This allows us to take $K$ to be the $(q,c)$-cable of $T_{r,b}$. Note that $r>q$ implies $b>1$, so $K$ is cable of a non-trivial torus knot. It is then a straightforward calculation, using Proposition~\ref{prop:torussurgerycalculation}, that for such a $K$ we have
\[S_{p/q}^3(T_{r,s})\cong S_{p/q}^3(K)\cong S_{p/q^3}^3(T_{r,b}) \cong S^2(0; \frac{1-q^2}{r}, \frac{ \mp q^3}{s}, \frac{\mp q}{b}),\]
as required.
\end{proof}

Now we consider the possibility that two torus knots share a surgery. The following generalizes \cite[Proposition~2.4]{Ni14characterizing}.
\begin{lem}\label{lem:toruscase}
If $S_{p/q}^3(T_{r,s})\cong S_{p/q}^3(T_{a,b})$, for some $p/q\in \Q$, then $T_{r,s}=T_{a,b}$ or $p/q\in \{rs\pm 1\}$.
\end{lem}
\begin{proof}
Suppose that $Y\cong S_{p/q}^3(T_{r,s})\cong S_{p/q}^3(T_{a,b})$. Without loss of generality we may assume that $r,s>1$.

If $p/q=rs$, then $Y$ is reducible. In this case we must have $ab=rs$ and
\[Y\cong S^3_{r/s}(U)\#S^3_{s/r}(U)\cong S^3_{a/b}(U)\#S^3_{b/a}(U).\]
Therefore $T_{r,s}=T_{a,b}$, in this case.

If $|p-rsq|>1$, then $Y$ is a Seifert fibred space with base orbifold
\[S^2(r,s,|p-rsq|)\cong S^2(|a|,|b|,|p-abq|).\]
It follows that we can assume $r=a$. By applying the Casson-Walker invariant as in \eqref{eq:secondderiv}, we have
\begin{equation*}
(a^2-1)(b^2-1)=(r^2-1)(s^2-1),
\end{equation*}
which implies that $s=|b|$. Thus, if $T_{a,b}\ne T_{r,s}$, then $T_{a,b}=T_{r,-s}$. However, $T_{a,b}=T_{r,-s}$ implies that $|p-rsq|=|p+rsq|$, and hence that $p=0$. It is easy to check, using Proposition~\ref{prop:torussurgerycalculation}, that $S_0^3(T_{r,s})\not\cong S_0^3(T_{-r,s})$.

Thus it only remains to consider the case that $|p-rsq|=1$ and $q>1$. In this case, we also have $|p-abq|=1$ and $Y$ is the lens space \cite{Moser71elementary}
\[Y\cong L(p,qr^2) \cong L(p,qa^2).\]
We can assume that $1<a<b$ and $1<r<s$. For $L(p,qr^2)$ to admit an orientation preserving homeomorphism to $L(p,qa^2)$ we must have either $qr^2 \equiv qa^2 \bmod p$ or $q^2r^2a^2\equiv 1 \bmod p$. As $qr^2<p$ and $qa^2<p$, we see that $qr^2 \equiv qa^2 \bmod p$ implies that $r=a$ and hence that $T_{r,s}=T_{a,b}$. Thus we can assume that $q^2r^2a^2\equiv 1 \bmod p$. Consider the continued fraction expansion
$s/r=[a_0, \dots , a_k]^+$, where $a_k\geq 2$ and $a_i\geq 1$ for $0\leq i\leq k-1$. Using some standard identities for continued fractions one can show that\footnote{The stated equalities follow from applications of the following identities, all of which admit straight-forward proofs:
\begin{enumerate}
\item $\frac{q_n}{q_{n-1}}=[c_n, \dots, c_1]^+$,
\item $p_n q_{n-1}-q_n p_{n-1} = (-1)^{n+1}$, and
\item for any $x\in \Q$, $[c_0, \dots, c_n,x]^+=\frac{p_nx+p_{n-1}}{q_n x + q_{n-1}}$,
\end{enumerate}
where $\frac{p_n}{q_n}=[c_0, \dots, c_n]^+$ denotes the $n$th convergent of a continued fraction with $c_i\geq 1$ for all $i$.}
\[\frac{p}{qr^2}=
\begin{cases}
[a_0,a_1, \dots, a_k,q-1,1, a_k-1, a_{k-1}, \dots, a_1]^+ &\text{if $p=qrs+(-1)^k$}\\
[a_0,a_1, \dots,a_{k-1}, a_k-1,1,q-1, a_k, \dots, a_1]^+ &\text{if $p=qrs-(-1)^k$.}
\end{cases}\]
Since both of these expansions have odd length, we see that if $q^2r^2a^2\equiv 1 \bmod p$, then $\frac{p}{qa^2}$ must have have an expansion of the form\footnote{Here we are using the fact that if $\frac{p}{q}=[c_0, \dots , c_n]^+$, then $\frac{p}{q'}=[c_n, \dots , c_0]^+$, where $qq'\equiv (-1)^n \bmod p$.}
\begin{equation}\label{eq:contfrac1}
\frac{p}{qa^2}=
\begin{cases}
[a_1, \dots, a_{k-1},a_k-1,1,q-1, a_k, \dots, a_0]^+ &\text{if $p=qrs+(-1)^k$}\\
[a_1, \dots, a_k,q-1,1, a_k-1, a_{k-1}, \dots, a_0]^+ &\text{if $p=qrs-(-1)^k$.}
\end{cases}\end{equation}
However, if we consider the continued fraction expansion $b/a=[b_0, \dots , b_l]^+$, where $b_l\geq 2$ and $b_i\geq 1$ for $0\leq i\leq l-1$, then
\begin{equation}\label{eq:contfrac2}
\frac{p}{qa^2}=
\begin{cases}
[b_0,b_1, \dots, b_l,q-1,1, b_l-1, b_{l-1}, \dots, b_1]^+ &\text{if $p=qab+(-1)^l$}\\
[b_0,b_1, \dots,b_{l-1}, b_l-1,1,q-1, b_l, \dots, b_1]^+ &\text{if $p=qab-(-1)^l$.}
\end{cases}\end{equation}
Since $\frac{p}{qa^2}$ admits a unique continued fraction expansion of odd length with every coefficient strictly positive, we see that the two continued fraction expansions in \eqref{eq:contfrac1} and \eqref{eq:contfrac2} must be the same. Comparing the lengths of these two expansions for $\frac{p}{qa^2}$ shows that $l=k$. Comparing the coefficients individually and using the assumption that $a_k,b_k>1$ soon shows that $\frac{s}{r}=\frac{b}{a}=[\underbrace{q,\dots,q}_{k+1}]^+$. Altogether, this shows that $T_{r,s}=T_{a,b}$ unless $p=rs\pm1$.
\end{proof}
\begin{rem}
When combined with the cyclic surgery theorem of Culler, Gordon, Luecke and Shalen \cite{cglscyclic}, Lemma~\ref{lem:toruscase} implies that for any $q>1$, any slope of the form $p/q= rs\pm \frac{1}{q}$ is a characterizing slope for $T_{r,s}$.
\end{rem}

Using results of Agol \cite{Agol00BoundsI}, Lackenby \cite{Lackenby03Exceptional}, and Cao and Meyerhoff \cite{Cao01cusped}, Ni and Zhang give a restriction on exceptional slopes of a hyperbolic knot.
\begin{prop}[Lemma~2.2, \cite{Ni14characterizing}] \label{prop:hyperbolicbound}
Let $K\subseteq S^3$ be a hyperbolic knot. If
\[|p|\geq \frac{43}{4}(2g(K)-1)\]
then $S^3_{p/q}(K)$ is a hyperbolic manifold.
\end{prop}
This is the final ingredient we require for the proof of the main theorem.

{\customthm{Theorem~\ref{thm:toruscharslopes}}
For $s>r>1$ and $q\geq 2$ let $K$ be a knot such that $S_{p/q}^3(K)\cong S_{p/q}^3(T_{r,s})$. If $p$ and $q$ satisfy at least one of the following:
\begin{enumerate}[(i)]
\item $p \leq \min \{-\frac{43}{4}(rs-r-s),-32q\}$,
\item $p\geq \max\{\frac{43}{4}(rs-r-s), 32q+ 2q(r-1)(s-1)\}$, or
\item $q\geq 9$,
\end{enumerate}
then we have either $(a)$ $K=T_{r,s}$, or $(b)$ $K$ is a cable of a torus knot, in which case $q=\lfloor s/r\rfloor$, $p=\frac{r^2 q^4 -1}{q^2-1}$, $s=\frac{rq^3 \pm 1}{q^2-1}$ and $r>q$.}

\begin{proof}
Let $K$ be a knot such that $S_{p/q}^3(K)\cong S_{p/q}^3(T_{r,s})$ for some $p/q$ with $q\geq 2$ such that at least one of conditions $(i),(ii)$ or $(iii)$ are satisfied.
Since $4q^2+12-2q<32q$ for $q\leq 8$ and $g(T_{r,s})=\frac{(r-1)(s-1)}{2}$, Theorem~\ref{thm:technical2} shows that either (a) $q\geq 9$ or (b) $K$ is fibred with $g(K)=g(T_{r,s})$ and $|p|\geq \frac{43}{4} (2g(K)-1)$.

According to Thurston, every knot in $S^3$ is either a hyperbolic knot, a satellite knot or a torus knot \cite{Thurston823DKleinanGroups}. We consider each of these possibilities in turn.

Given Proposition~\ref{prop:hyperbolicbound} and the fact that any exceptional surgery on a hyperbolic knot in $S^3$ must satisfy $q\leq 8$ \cite{lackenby13exceptional}, we see that $K$ cannot be a hyperbolic knot.

If $K$ is a satellite knot, then there is an incompressible torus $R\subseteq S^3 \setminus K$. This bounds a solid torus $V\subseteq S^3$ which contains $K$. Let $K'$ be the core of the solid torus $V$. By choosing $R$ to be innermost, we may assume that $K'$ is not a satellite. This means that $K'$ is either a torus knot or a hyperbolic knot. Since $S^3_{p/q}(T_{r,s})$ is irreducible and does not contain an incompressible tori, it follows from the work of Gabai that $V_{p/q}(K)$ is again a solid torus and that $K$ is either a 1-bridge knot or a torus knot in $V$ and $S_{p/q}^3(K)\cong S_{p/(qw^2)}^3(K')$, where $w>1$ is the winding number of $K$ in $V$ \cite{gabai89solidtori}. Moreover, as $q>1$, it follows that $K$ is a torus knot in $V$.

If $q\geq 9$, then the distance bound on exceptional surgeries shows that $K'$ is not hyperbolic. If $q\leq 8$, then $K$ is fibred, implying that $K'$ is also fibred \cite{Hirasawa2008Fiber}. It follows, for example by considering the degrees of $\Delta_K(t)$ and $\Delta_{K'}(t)$, that $g(K')< g(K)$. Therefore, if $q\leq 8$, then we have $|p|> \frac{43}{4} (2g(K')-1)$ and Proposition~\ref{prop:hyperbolicbound} shows that $K'$ is not hyperbolic. Thus we can assume that $K'$ is a torus knot. Thus, we have shown that if $K$ is a satellite, then it is a cable of a torus knot. In this case Proposition~\ref{prop:cablenoncharslopes} applies to give the desired conclusions on $p,q,r$ and $s$.

If $K$ is a torus knot, then Lemma~\ref{lem:toruscase} shows that $K=T_{r,s}$, as required.
\end{proof}

{\customthm{Corollary~\ref{cor:intcase}}
The knot $T_{r,s}$ with $r,s>1$ has only finitely many non-characterizing slopes which are not negative integers.}
\begin{proof}
This follows from Theorem~\ref{thm:toruscharslopes} and the results of \cite{mccoy2014sharp} which show that any slope
\[\frac{p}{q}\geq \frac{43}{4}(rs-r-s)\]
is a characterizing slope for the torus knot $T_{r,s}$ with $r,s>1$.
\end{proof}

\bibliographystyle{alpha}
\bibliography{cone}
\end{document}